\documentclass[11pt]{article}
\usepackage{amsmath,amssymb,amsthm,mathtools,diagrams,theoremref,a4wide}

\DeclareMathOperator{\tr}{tr}

\DeclareMathOperator{\Spec}{Spec}

\DeclareMathOperator{\Aut}{Aut}

\DeclareMathOperator{\Mod}{-Mod}

\DeclareMathOperator{\FI}{FI}

\DeclareMathOperator{\Conf}{Conf}
\DeclareMathOperator{\Frob}{Frob}
\DeclareMathOperator{\UConf}{UConf}

\DeclareMathOperator{\Gal}{Gal}
\DeclareMathOperator{\Bl}{Bl}
\DeclareMathOperator{\Poly}{Poly}

\newtheorem{thm}{Theorem}[section]
\newtheorem{prop}[thm]{Proposition}
\theoremstyle{definition}
\newtheorem{defn}[thm]{Definition}
\newtheorem{rem}[thm]{Remark}
\begin{document}

\title{\(\FI_G\)-modules and arithmetic statistics}
\author{Kevin Casto\footnote{Partially supported by NSF grants DMS-1105643 and DMS-1406209.}}
\maketitle

\begin{abstract}
This is a sequel to the paper \cite{Ca1}. Here, we extend the methods of Farb-Wolfson \cite{FW2} using the theory of \(\FI_G\)-modules to obtain stability of equivariant Galois representations of the \'etale cohomology of orbit configuration spaces. We establish subexponential bounds on the growth of unstable cohomology, and then use the Grothendieck-Lefschetz trace formula to obtain results on arithmetic statistics for orbit configuration spaces over finite fields. In particular, we show that the average value, across polynomials over \(\mathbb{F}_q\), of certain Gauss sums over their roots, stabilizes as the degree goes to infinity.
\end{abstract}

\section{Introduction}
In \cite{Ca1}, the author considered the orbit configuration space \(\Conf_n^G(M)\) associated to a free action of a finite group \(G\) on a manifold \(M\),
\[ \Conf_n^G(M) := \{(m_i) \in M^n \mid m_i \ne g m_j \,\forall g \in G\}. \]
The cohomology \(H^i(\Conf_n^G(M); \mathbb{Q})\) is a representation of \(W_n := G^n \rtimes S_n\). We showed that, if we consider all \(n\) at once, then \(H^i(\Conf^G(M); \mathbb{Q})\) forms a module over the category \(\FI_G\) introduced by \cite{SS2}. Furthermore, we proved \cite[Thm 3.1]{Ca1} that this \(\FI_G\)-module is finitely generated, when mild conditions are \(M\) are satisfied (e.g., \(\dim H^*(M; \mathbb{Q}) < \infty\)).

This finite generation implies the following. Given a partition-valued function \(\underline \lambda\) on the irreducible representations of \(G\) with \(\|\underline \lambda\| = n\), let \(L(\underline \lambda)\) be the associated irreducible representation of \(W_n\). Let \(c(G)\) be the set of conjugacy classes of \(G\). Define an \emph{\(\FI_G\) character polynomial} to be a polynomial in \(c(G)\)-labeled cycle-counting functions. This is a class function on \(W_n\). (We will define this and other terminology more precisely in \S2.2).
\begin{thm}[\cite{Ca1}]
Let \(M\) be a connected manifold with a free action of a finite group \(G\). Assume that \(\dim M \ge 2\) and that each \(\dim H^i(M; \mathbb{Q}) < \infty\). Then:\begin{enumerate}
\item The characters of \(H^i(\Conf_n^G(M); \mathbb{C})\) are given by a single character polynomial for all \(n \gg 0\).
\item The multiplicity of each irreducible \(W_n\)-representation in \(H^i(\Conf_n^G(M); \mathbb{C})\) is eventually independent of \(n\), and \(\dim H^i(\Conf_n^G(M); \mathbb{C})\) is given by a single polynomial for all \(n \gg 0\).
\end{enumerate}
\end{thm}

In this paper we consider the case where we replace the manifold \(M\) with a scheme \(X\) over \(\mathbb{Z}[1/N]\), or more generally \(\mathcal{O}_k[1/N]\) for a number field \(k\), with an algebraic action of \(G\). Here the \(\mathbb{C}\)-points \(X(\mathbb{C})\) take the place of \(M\). However, now we can also consider the points \(X(\mathbb{F}_q)\), for \(\mathbb{F}_q\) a residue field of \(\mathcal{O}_k\).

We generalize the results of Farb-Wolfson \cite{FW2} on \(\Conf_n(X)\) to the orbit configuration space \(\Conf^G_n(X)\). The first result is on \emph{\'etale representation stability} in the sense of \cite{FW2}.

\begin{thm}[\bfseries \'Etale representation stability for orbit configuration spaces]
Let \(X\) be a smooth scheme over \(\mathcal{O}_k[1/N]\) with geometrically connected fibers. Let \(G\) be a finite group acting freely on \(X\), such that \(X\) is smoothly compactifiable as a \(G\)-scheme. Let \(K\) be either a number field or an unramified function field over \(\mathcal{O}_k[1/N]\). For each \(i \ge 0\), the \(\Gal(\overline{K}/K)\)-\(\FI_G\)-module \(H^i_{\acute et}(\Conf^G(X)_{/\overline{K}}; \mathbb{Q}_l)\) is finitely generated.
\end{thm}

The next result concerns bounds on \(H^i(\Conf^G_n(X))\) as \(i\) varies, which are necessary to ensure convergence of the point-counts we are interested in.

\begin{thm}[\bfseries Orbit configuration spaces have convegent cohomology]
\thlabel{conf_convg}
Let \(X\) be a connected orientable manifold of dimension at least 2 with \(\dim H^*(M; \mathbb{Q})\) finite-dimensional. Let \(G\) be a finite group acting freely on \(M\). Then for each character polynomial \(P\), the inner product  \(|\langle P, H^i(\Conf^G_n(X))\rangle|\) is bounded subexponentially in \(i\) and uniformly in \(n\).
\end{thm}

Finally, we use the Grothendieck-Lefschetz trace formula, along with Theoresm 1.2 and 1.3, to obtain the following result on arithmetic statistics.

\begin{thm}[\bfseries Arithmetic statistics for orbit configuration spaces stabilize]
\thlabel{arith_stats}
Let \(X\) be a smooth quasiprojective scheme over \(\mathcal{O}_k[1/N]\) with geometrically connected fibers. Let \(G\) be a finite group acting freely on \(X\), such that \(X\) is a smoothly compactifiable \(G\)-scheme. Then for any \(\FI_G\) character polynomial \(P\),
\begin{equation} \label{groth}
 \lim_{n \to \infty} q^{-n \dim X} \sum_{y \in \UConf_n(X/G)(\mathbb{F}_q)} P(\sigma_y) = \sum_{i = 0}^\infty (-1)^i \tr\left(\Frob_q : \langle H^i(\Conf^G(X(\mathbb{C})); L), P\rangle\right)
\end{equation}
In particular both the limit on the left and the series on the right converge, and they converge to the same limit. 
\end{thm}
\subsection*{Gauss Sums}
For the specific case \(G = \mathbb{Z}/d\mathbb{Z}\), the resulting automorphism groups \(W_n = \left(\mathbb{Z}/d\mathbb{Z}\right)^n \wr S_n\) are the so-called \emph{main series of complex reflection groups}. These directly generalize the Weyl groups of type \(BC_n\) of Wilson's paper, for which \(d = 2\). In particular, we can apply \thref{arith_stats} to the action of \(\mathbb{Z}/d\mathbb{Z}\) on \(\mathbb{C}^*\) by rotation. In this case, the sum on the left-hand side of (\ref{groth}) is over \(\UConf_n(\mathbb{G}_m)(\mathbb{F}_q) = \Poly_n(\mathbb{F}_q^*)\), the space of square-free polynomials over \(\mathbb{F}_q\) that do not have 0 as a root. We thus obtain the following result, generalizing the work of Church-Ellenberg-Farb \cite{CEF2} (see \S6 for definitions).
\begin{thm}[\bfseries Gauss sums for \(\Conf^G(\mathbb{C}^*)\) stabilize]
\thlabel{gauss}
For any prime power \(q\), any \(d \mid q-1\), and any character polynomial \(P\) for \(\FI_{\mathbb{Z}/d\mathbb{Z}}\),
\begin{equation} \label{gauss2}
\lim_{n \to \infty} q^{-n} \sum_{f \in \Poly_n(\mathbb{F}_q^*)} P(f) = \sum_{i = 0}^\infty (-1)^i \frac{\langle P^*, H^i(\Conf^{\mathbb{Z}/d\mathbb{Z}}(\mathbb{C}^*); \mathbb{C})\rangle}{q^i}
\end{equation}
In particular both the limit on the left and the series on the right converge, and they converge to the same limit. 
\end{thm}

\thref{gauss} essentially says that the \emph{average value} of certain Gauss sums across all polynomials in \(\Poly_n(\mathbb{F}^*_q)\) always converges to the series in \(q^{-1}\) on the right. For example, let \(\chi\) be a character of \(\mathbb{Z}/(q-1)\mathbb{Z}\). Define the character polynomial \(X_i^\chi := \sum_{g \in G} \chi(g) X^g_i\). Then
\begin{gather} \begin{aligned} \label{X1_pcount}
\lim_{n \to \infty} q^{-n} &\sum_{f \in \Poly_n(\mathbb{F}_q^*)} \sum_{\substack{\alpha \ne \beta \in \mathbb{F}_q\\f(\alpha) = f(\beta) = 0}} \chi(\alpha)\chi(\beta)^{-1}  \\
&= \sum_i (-1)^i \frac{\langle X_1^{\overline{\chi}} X_1^{\chi}, H^i(\Conf^{\mathbb{Z}/(q-1)\mathbb{Z}}(\mathbb{C}^*); \mathbb{Q}(\zeta_{q-1}))\rangle}{q^i} = -\frac{1}{q} + \frac{5}{q^2} + \cdots
\end{aligned} \end{gather}
That is, the average value of the Gauss sum obtained by applying \(\chi\) to each quotient of pairs of linear factors of \(f\), across all \(f \in \Poly_n(\mathbb{F}_q^*)\), is equal to the series on the right-hand side of (\ref{X1_pcount}) obtained by looking at the inner product of the character polynomial \(X_1^{\overline{\chi}} X_1^{\chi}\) with \(H^i(\Conf^{\mathbb{Z}/(q-1)\mathbb{Z}}(\mathbb{C}^*); \mathbb{Q}(\zeta_{q-1}))\). As another example, suppose \(q\) is odd and let \(\psi = \left(\frac{-}{q^2}\right)\) be the Legendre symbol in \(\mathbb{F}_q^2\), which is \(1\) or \(-1\) according to whether its argument is a square or nonsquare in \(\mathbb{F}_{q^2}\). Then
\begin{gather} \begin{aligned} \label{X2_pcount}
\lim_{n \to \infty} q^{-n} &\sum_{f \in \Poly_n(\mathbb{F}_q^*)} \sum_{\substack{p \mid f \\ \deg(p) = 2}} \psi(\text{root}(p))  \\
&= \sum_i (-1)^i \frac{\langle X_2^{\psi}, H^i(\Conf^{\mathbb{Z}/(q-1)\mathbb{Z}}(\mathbb{C}^*); \mathbb{Q}(\zeta_{q-1}))\rangle}{q^i} = -\frac{1}{q} + \frac{3}{q^2} + \cdots
\end{aligned} \end{gather}
where \(\text{root}(p)\) denotes a root of \(p\), an irreducible degree 2 factor of \(f\), lying in \(\mathbb{F}_{q^2}\); the value \(\psi(\text{root}(p))\) turns out not to depend on the choice of root. Thus, (\ref{X2_pcount}) says that the average value of the Gauss sum obtained by applying \(\psi\) to the \emph{quadratic} factors of \(f\), across all \(f \in \Poly_n(\mathbb{F}^*_q)\), is equal to the series on the right obtained by looking at the inner product of the character polynomial \(X_2^{\psi}\) with \(H^i(\Conf^{\mathbb{Z}/(q-1)\mathbb{Z}}(\mathbb{C}^*); \mathbb{Q}(\zeta_{q-1}))\).

\begin{rem}
Just as in \cite[\S4.3]{CEF2}, it is quite likely that one can compute the left-hand side of (\ref{gauss2}) using twisted L-functions of the form
\[ L(P,s) = \sum_n \sum_{f \in \Poly_n(\mathbb{F}_q^*)} P(f) q^{-ns} \]
or by other analytic methods. Zeev Rudnick sketched such an argument for the case of (\ref{X1_pcount}) in a private communication. However, \thref{gauss} gives a topological interpretation to the left-hand side of (\ref{gauss2}). More to the point, the fact that representation stability holds for orbit configuration spaces is what suggests \thref{gauss} in the first place. In fact, the bridge to topology provided by the Grothendieck-Lefschetz trace formula gives further motivation to study such L-functions. One can often go in the other direction, and prove representation stability by means of counting points over finite fields. Chen \cite{WCh2} has done this for the usual configuration spaces, and it would be interesting to investigate this for orbit configuration spaces as well.
\end{rem}

\subsection*{Related work}
After distributing this paper, we learned of new work of Rolland-Wilson \cite{RW2} investigating the specific case of the type B/C hyperplane arrangement. This corresponds to the case \(d = 2\) in \thref{gauss}. Their proof of the necessary convergence results are by completely different methods. Our proof of the general result \thref{conf_convg} is by the same argument as Farb-Wolfson \cite{FW2}, and we prove the stronger result \thref{convg} for branched covers of \(\mathbb{P}^1\) by a topological argument. Rolland-Wilson's result, \cite[Thm 3.8]{RW2}, is proven using a novel graph-theoretic argument. It would be very interesting to see if their method could be strengthened to prove the stronger polynomial bounds of our \thref{convg}.

\subsection*{Acknowledgements}
I would like to thank Nir Gadish, Joel Specter, Weiyan Chen, Jesse Wolfson, and Sean Howe for helpful conversations. I would like to thank Zeev Rudnick for pointing out mistakes in some calculations in the initial version of this paper, and for other helpful comments. Also, I am very grateful to my advisor, Benson Farb, for all his guidance throughout the process of working on, writing, and revising this paper.

\section{\'Etale homological stability and orbit configuration spaces}

In this section we extend the theory of \'etale representation stability, developed by Farb-Wolfon \cite{FW2} for the category \(\FI\), to the category \(\FI_G\). We then apply this theory to orbit configuration spaces.

\subsection{\'Etale representation stability}
If \(G\) is a group and if \(R\) and \(S\) are sets, define a \emph{\(G\)-map} \((a, (g_i)): R \to S\) to be a pair \(a: R \to S\) and \((g_i) \in G^R\). If \((b, (h_j)): S \to T\) is another \(G\)-map, their composition is \((b \circ a, (g_i \cdot h_{a(i)}))\). Let \(\FI_G\) be the category with objects finite sets and morphisms \(G\)-maps with the function \(a\) injective; in particular, the category \(\FI = \FI_{\text{triv}}\). An \emph{\(\FI_G\)-module} \(V\) is a functor from \(\FI_G\) to \(k\Mod\) for some ring \(k\). We have
\[\Aut_{\FI_G}([n]) = G^n \rtimes S_n =: W_n\]
so that if \(V\) is an \(\FI_G\)-module, then each \(V_n\) is a \(W_n\)-module. An \(\FI_G\)-module is finitely-generated if it has a finite set of elements not contained in any proper submodule.

In \cite{Ca1}, the author developed a theory of representation stability for \(\FI_G\)-modules, extending the theory developed for \(\FI\) by Church-Ellenberg-Farb \cite{CEF} and \(\FI_{\mathbb{Z}/2\mathbb{Z}}\) by Wilson \cite{Wi2}, and building on the work of \cite{SS2} who introduced \(\FI_G\). We thus have the following results about finitely-generated \(\FI_G\)-modules, many of which were also obtained by Gadish \cite{Ga2}.

\begin{thm}[\bfseries Structural properties of finitely-generated \(\FI_G\)-modules]
\thlabel{fig_props}
Let \(G\) be a finite group and let \(V\) be an \(\FI_G\)-module over a field \(k\) of characteristic zero. If \(V\) is finitely-generated then:
\begin{description}

\item [Representation stability (\cite{GL3}):] \(V\) satisfies representation stability in the sense of \cite{CF} for the sequence \(\{V_n\}\) of \(W_n\)-modules.
\item [Isomorphism of trivial isotypics (\cite{Ca1}):] For all sufficiently large \(n\), the map \(V_n \to V_{n+1}\) given by averaging the structure maps induces an isomorphism
\[ V_n^{W_n} \xrightarrow \simeq V_{n+1}^{W_{n+1}} \]
\item [Tensor products (\cite{Ca1}):] If \(W\) is another finitely-generated \(\FI_G\)-module, and \(G\) is finite, then \(V \otimes W\) is finitely-generated.
\item [Character polynomials (\cite{Ca1}):] If \(k\) is a splitting field for \(G\), then there is a character polynomial \(P_V\) so that
\[ \chi_{V_n} = P_V \text{ for } n \gg 0. \]
\end{description}
\end{thm}

Farb-Wolfson define a notion of \emph{\'etale represention stability} for a co-FI scheme \(Z\) over \(\mathbb{Z}[1/N]\). They show that if there is a uniform way of normally compactifying the \(Z_n\), then the base-change maps commute with the induced FI maps. This allows them to pass from knowing that \(H^i(Z(\mathbb{C}))\) is a finitely-generated FI-module to knowing that \(H^i_{\acute et}(Z_{/\overline{\mathbb{F}}_p}; \mathbb{Q}_l)\) is a finitely-generated \emph{\(\Gal(\overline{K}/K)\)-FI-module}. Since Fulton-Macpherson constructed a normal compactification of \(\Conf(X)\), this allows them to conclude that \(\Conf(X)\) satisfies \'etale representation stability.

We want to generalize this to the orbit configuration space \(\Conf^G(X)\). However, to do so, we need to allow for the possibility that \(Z\) is only defined over some finite Galois extension of \(\mathbb{Q}\). For example, the orbit configuration space
\[ \Conf^{\mathbb{Z}/d\mathbb{Z}}_n(\mathbb{G}_m) = \{(x_i) \in \mathbb{G}_m^n \mid x_i \ne \zeta^m x_j\}\]
where \(\zeta\) is a primitive \(d\)-th root of unity and the group \(G = \mathbb{Z}/d\mathbb{Z}\) acts on \(\mathbb{G}_m\) by multiplication by \(\zeta\), is naturally defined over \(\mathbb{Z}[\zeta]\). Thus, rather than reducing modulo primes \(p\), we now have to reduce modulo prime ideals \(\mathfrak{p}\) of \(\mathcal{O}_k[1/N]\). We then have \(\mathcal{O}_k/\mathfrak{p} = \mathbb{F}_{q}\), where \(q = p^f\) with \(f\) the inertia degree of \(\mathfrak{p}\) over \(p\).

We therefore make the following definition, following \cite{FW2}.

\begin{defn}[\bfseries Smoothly compactifiable \(I\)-scheme over \(k\)]
Let \(I\) be a category and \(k\) a Galois number field. A \emph{smooth \(I\)-scheme over \(k\)} is a functor \(Z: I \to \mathbf{Schemes}\) consistings of smooth schemes over \(\mathcal{O}_k[1/N]\) for some \(N\) independent of \(i\). A smooth \(I\)-scheme is \emph{smoothly compactifiable at \(\mathfrak{p} \nmid N\)} if there is a smooth projective \(I\)-scheme \(\overline{Z}\) and a natural transformation \(Z \to \overline{Z}\) so that for all \(i \in I\), \(Z_i \to \overline{Z}_i\) is an open embedding and \(\overline{Z}_i - Z_i\) is a normal crossings divisor with good reduction at \(\mathfrak{p}\).
\end{defn}

We then have the following, generalizing \cite[Thm 2.6]{FW2}:

\begin{thm}[\bfseries Base change for \(I\)-schemes over \(k\)] \thlabel{basechange}
Let \(l\) be a prime, and \(Z: I \to \mathbf{Schemes}\) be a smooth \(I\)-scheme over \(k\) which is smoothly compactifiable at \(\mathfrak{p} \nmid N \cdot l\). Then for all morphisms \(i \to j\) in \(I\), the following diagram of ring homomorphisms commutes:
\[ \begin{diagram}
H^*_{\acute et}(Z_{j/\overline{\mathbb{F}}_q}; \mathbb{Z}_l) & \rTo^{\sim} & H^*_{sing}(Z_j(\mathbb{C}); \mathbb{Z}_l) \\
\dTo && \dTo \\
H^*_{\acute et}(Z_{i/\overline{\mathbb{F}}_q}; \mathbb{Z}_l) & \rTo^{\sim} & H^*_{sing}(Z_i(\mathbb{C}); \mathbb{Z}_l)
\end{diagram}\]
\end{thm}
\begin{proof}
The proof of \cite[Thm 2.6]{FW2} applies verbatim, since as they say, theirs is essentially an \(I\)-scheme version of \cite[Thm 7.7]{EVW}, which is more than general enough to include our extension to number fields. Just replace \(p\) with \(\mathfrak{p}\) and \(\mathbb{F}_p\) with \(\mathbb{F}_q\) everywhere in the proof of \cite[Thm 2.6]{FW2}.
\end{proof}

Now, let \(Z\) be a co-\(\FI_G\) scheme over \(k\) and \(l\) a prime. For each \(i \ge 0\) the \'etale cohomology \(H^i_{\acute et}(Z_{/\overline{k}}, \mathbb{Q}_l)\) is an \(\FI_G\)-module equipped with an action of \(\Gal(\overline{k}/k)\) commuting with the \(\FI_G\) action. Following \cite{FW2}, we call such an object a \emph{\(\Gal(\overline{k}/k)\)-\(FI_G\)-module}. Likewise, if \(\mathfrak{p} \nmid N \cdot l\) is a prime ideal of \(k\), and we put \(\mathbb{F}_q = \mathcal{O}_k/\mathfrak{p}\), then \(H^i_{\acute et}(Z_{\overline{\mathbb{F}}_q}, \mathbb{Q}_l)\) is a \(\Gal(\overline{\mathbb{F}}_q/\mathbb{F}_q)\)-FI-module. 

For \(K\) a number field or finite field of characteristic \(p\), a \(\Gal(\overline{K}/K)\)-\(\FI_G\)-module is finitely generated if it has a finite set of element not contained in any proper \(\Gal(\overline{K}/K)\)-FI-module. We can therefore use \thref{basechange} to pass from knowing finite generation of \(H^i(Z(\mathbb{C}))\) as an \(\FI_G\)-module to finite generation of \(H^i(Z_{/K}; \mathbb{Q}_l)\) as a \(\Gal(\overline{K}/K)\)-\(\FI_G\)-module.

\subsection{Stability of orbit configuration spaces}
Let \(X\) be a scheme over \(\mathcal{O}_k[1/N]\) with a free algebraic action of a finite group \(G\). The \emph{orbit configuration space} \(\Conf^G_n(X)\) is defined as the functor of points:
\[ \Conf^G_n(X)(R) = \{(x_i) \in X(R)^n \mid x_i \ne g x_j \; \forall g \in G\}\]
for any \(\mathcal{O}_k[1/N]\)-algebra \(R\). Thus \(\Conf^G_n(X)\) has the structure of a scheme over \(\mathbb{Z}[1/N]\). If \(X\) is smooth, then \(\Conf^G_n(X)\) is smooth.

The group \(W_n := G^n \rtimes S_n\) acts freely on \(\Conf^G_n(X)\), where the copies of \(G\) act on each coordinate and \(S_n\) permutes the coordinates. In fact, \(\Conf^G(X)\) forms a co-\(\FI_G\)-scheme: for an injection \(a: [m] \hookrightarrow [n]\) and \(\vec g \in G^m\), the action on \((x_i) \in \Conf^G_n(X)\) is given by
\[ (\vec g, a) (x_i) = ( g_i \cdot x_{a(i)}) \]
It thus follows that \(H^i(\Conf^G(X))\) is an \(\FI_G\)-module. In \cite[Thm 3.1]{Ca1}, we proved that \(H^i(\Conf^G(X)\) is finitely generated. We would like to use \thref{basechange} to deduce \'etale representation stability for \(\Conf^G(X)\). To do so, we need to know that \(\Conf^G(X)\) is smoothly compactifiable at all but finitely many primes. Luckily, there are general constructions of smooth compactifications for the complement of an arrangement of subvarieties. The basic construction is due to Macpherson-Procesi \cite{MP} in the complex-analytic setting, generalizing Fulton-Macpherson's \cite{FM} compactification of \(\Conf_n(X)\). This work was extended to the scheme-theoretic setting by Hu \cite{Hu} and Li \cite{Li}. Following Li, define the \emph{wonderful compactification} \(\overline{\Conf^G_n}(X)\) to be the closure of the image of the embedding
\begin{equation} \Conf^G_n(X) \hookrightarrow \prod_{\substack{i \ne j \\ g \in G}} \Bl_{\Delta_{i,g,j}}(X^n)
\end{equation}
of the product of the blowups of \(X^n\) along the diagonals \(\Delta_{i,g,j} := \{(x_i) \mid x_i = g x_j\}\).

To carry this out, we first need to know that \(X\) itself can be smoothly compactified in a manner consistent with the action of \(G\). For a smooth scheme \(X\) equipped with an action of a finite group \(G\), say that \(X\) is a \emph{smoothly compactifiable \(G\)-scheme} if there is a scheme \(\overline{X}\) with an action of \(G\) and a \(G\)-equivariant embedding \(X \hookrightarrow \overline{X}\) such that \(\overline{X} - X\) is a normal crossings divisor. Note that this is always possible in characteristic 0 by resolution of singularities, so this assumption is only needed to ensure such a compactification exists in finite characteristic.

We then have the following.
\begin{prop}[\bfseries \(\Conf^G(X)\) is smoothly compactifiable] \thlabel{cpctify}
Let \(X\) be a smooth scheme over \(\mathcal{O}_k[1/N]\) with a free action of a finite group \(G\) so that \(X\) is a smoothly compactifiable \(G\)-scheme. Then \(\overline{\Conf^G_n}(X)\) is a smoothly compactifiable co-\(\FI_G\)-scheme at any prime ideal \(\mathfrak{p} \nmid N\) of \(\mathcal{O}_k\).
\end{prop}
\begin{proof}
First, note that \(X^n\) is a smoothly compactifiable \(G^n\)-scheme, since \(X\) is smoothly compactifiable \(G\)-scheme. We want to use Li's \cite{Li} work on wonderful compactifications, but to do so we need to check that our arrangement satisfies his hypotheses, that the arrangement is a \emph{building set} in the sense of \cite[\S2]{Li}. This follows from the fact that \(G\) acts freely on \(X\), so that diagonals \(\Delta_{i,g,j}\) and \(\Delta_{i,h,j}\) are disjoint for \(g \ne h\).

By definition (3) of \(\overline{\Conf^G_n}(X)\) as the closure in the blowup, we see that \(W_n\) acts on \(\overline{\Conf^G_n}(X)\), and that \(\overline{\Conf^G}(X)\) forms a co-\(\FI_G\)-space. Furthermore, since \(X\) is a scheme over \(R := \mathcal{O}_k[1/N]\), then (3) defines \(\overline{\Conf^G}(X)\) as a scheme over \(R\). There is a natural open embedding \(\Conf^G(X) \hookrightarrow \overline{\Conf^G}(X)\), again just given by (3). To see that the complement \(\overline{\Conf^G}(X) - \Conf^G(X)\) is a normal crossings divisor, it is enough to check this at the geometric fibers over the points of \(R\). This holds by \cite[Thm 1.2]{Li}. Likewise, it is enough to check smoothness of \(\overline{\Conf^G}(X)\) at the geometric fibers over the points of \(R\). This holds by \cite[Thm 1.2]{Li}. Since \(\overline{\Conf^G}(X)\) is smooth and \(\overline{\Conf^G}(X) - \Conf^G(X)\) is a normal crossings divisor relative to \(\Spec R\), in the sense of \cite[Prop 7.7]{EVW}, we conclude that \(\overline{\Conf^G}(X) - \Conf^G(X)\) has good reduction at each prime \(\mathfrak{p}\) of \(R\).
\end{proof}

\begin{thm}[\bfseries \'Etale representation stability for orbit configuration spaces]
\thlabel{etale_repstab}
Let \(X\) be a smooth scheme over \(\mathcal{O}_k[1/N]\) with geometrically connected fibers. Let \(G\) be a finite group acting freely on \(X\), such that \(X\) is smoothly compactifiable as a \(G\)-scheme. Let \(K\) be either a number field or an unramified function field over \(\mathcal{O}_k[1/N]\). For each \(i \ge 0\), the \(\Gal(\overline{K}/K)\)-\(\FI_G\)-module \(H^i_{\acute et}(\Conf^G(X)_{/\overline{K}}; \mathbb{Q}_l)\) is finitely generated.
\end{thm}
\begin{proof}
Again, we follow \cite[Thm 2.8]{FW2}. Since \(X\) has geometrically connected fibers, \(X(\mathbb{C})\) is connected. Since \(X\) is smoothly compactifiable, \(H^*(X; \mathbb{Q})\) is finitely generated. Since \(X(\mathbb{C})\) is a complex manifold, it is orientable and has real dimension at least 2. Thus \(X\) satisfies the hypotheses of \cite[Thm 3.1]{Ca1}, and so \(H^i(X(\mathbb{C}); \mathbb{Q})\) is a finitely generated \(\FI_G\)-module.

If \(K\) is a number field, we conclude immediately from Artin's comparison theorem that \(H^i_{\acute et}(\Conf^G(X)_{/K}; \mathbb{Q}_l)\) is a finitely generated \(\Gal(\overline{K}/K)\)-\(\FI_G\)-module. If \(K\) is a finite field, by \thref{cpctify} \(\Conf^G(X)\) is a smoothly compactifiable co-\(\FI_G\)-scheme, and so by \thref{basechange}, the conclusion likewise follows.
\end{proof}

\section{Convergence}
Recall that the conjugacy classes of \(W_n = G^n \rtimes S_n\) are given by a cycle decomposition of size \(n\) (i.e., the usual conjugacy classes of \(S_n\)), but where each cycle is decorated by a conjugacy class of \(G\), where we write \(c(G)\) for the set of conjugacy classes of \(G\). A \emph{character polynomial for \(\FI_G\)} is a polynomial
\[ P \in k[\{X_i^C \mid C \in c(G)\}] \]
This determines a class function \(P_n : W_n \to k\) for all \(n\), where \(X_i^C\) counts the number of \(C\)-decorated \(i\)-cycles of a given conjugacy class of \(W_n\). Recall from \thref{fig_props} that if \(V\) is a finitely-generated \(\FI_G\)-module, then the characters \(\chi_{V_n}\) are eventually given by a single character polynomial for all large \(n\). Furthermore, for another character polynomial \(P\), we know that the inner product \(\langle P_n, V_n\rangle_{W_n}\) is eventually independent of \(n\). We put
\[ \langle P, V\rangle = \lim_{n \to \infty} \left\langle P, V_n \right\rangle_{W_n} \]
for the limiting multiplicity.

Thus if \(Z\) is a co-\(\FI_G\)-scheme that satisfies \'etale representation stability, there is a stable inner product \(\langle P, H^i_{\acute et}(Z; \mathbb{Q}_l)\rangle\), eventually independent of \(n\). However, for our applications we need to bound how these inner products grow in \(i\). The following proposition is helpful in doing so:

\begin{prop}\thlabel{convg} For any graded \(\FI_G\)-module \(V^*\), the following are equivalent: \begin{enumerate}
\item For each character polynomial \(P\), \(|\langle P_n, V^i_n \rangle|\) is bounded subexponentially in \(i\) and uniformly bounded in \(n\).
\item For every \(a\), the dimension \(\dim (V^i_n)^{W_{n-a}}\) is bounded subexponentially in \(i\) and uniformly bounded in \(n\).
\end{enumerate}
\end{prop}
\begin{proof}
\item This follows from \cite[Defn 3.12]{CEF2}. Their proof applies almost verbatim, we just need to replace \(S_n\) with \(W_n\) and ``cycles'' with ``\(c(G)\)-labeled cycles''.
\end{proof}

If a graded \(\FI_G\)-algebra \(V^*\) satisfies these two equivalent conditions, we say it is \emph{convergent}. \thref{conf_convg} thus states that, under appropriate conditions, the \(\FI_G\) algebra \(H^*(\Conf^G(X); \mathbb{Q})\) is convergent.

\begin{proof}[Proof of \thref{conf_convg}]
Farb-Wolfson \cite[Thm 3.4]{FW2} proved this for \(G = 1\) by using the Leray spectral sequence to reduce to convergence of \(H^*(X^n) \otimes H^*(\Conf(\mathbb{R}^d))\). We will use the same technique to reduce to their result.

Let \(A(G,d)_n\) be the graded commutative \(\mathbb{Q}\)-algebra generated by \(\{e_{a,g,b} \mid 1 \le a \ne b \le n, g \in G\}\), each of degree \(d-1\), modulo the following relations:
\begin{align*}
e_{a,g,b} &= (-1)^d e_{b,g^{-1},a} \\
e_{a,g,b}^2 &= 0 \\
e_{a,g,b} \wedge e_{b,h,c} &= (e_{a,g,b} - e_{b,h,c}) \wedge e_{a,gh,c}
\end{align*}
where the action of \((\vec h, \sigma) \in W_n\) is given by
\[ (\vec h, \sigma) \cdot e_{a,g,b} = e_{\sigma(a), h_a g h_b^{-1}, \sigma(b)} \]
If we let \(A(d)_n = A(1,d)_n\), then there is an embedding \(A(d)_n \hookrightarrow A(G,d)_n\) given by \(e_{a,b} \mapsto e_{a,e,b}\). In \cite[Thm 3.6]{Ca1}, we proved that \(A(G,d)_n\) is spanned by the \(G^n\)-translates of \(A(d)_n\).

In \cite[Thm 3.1]{Ca1}, we proved that \(H^*(\Conf^G_n(X); \mathbb{Q})\) is isomorphic as a graded \(W_n\)-representation, to a subquotient of
\[ H^*(X^n; \mathbb{Q}) \otimes A(G,d)_n. \]
Thus, it suffices to bound
\[ \left(H^*(X^n; \mathbb{Q}) \otimes A(G,d)_n\right)^{W_{n-a}} = \left(\left(H^*(X^n; \mathbb{Q}) \otimes A(G,d)_n\right)^{G^{n-a}}\right)^{S_{n-a}} \]
We know that \(\left(H^*(X^n; \mathbb{Q}) \otimes A(G,d)_n\right)^{G^{n-a}}\) is a subquotient of \(H^*(X^n; \mathbb{Q}) \otimes \left(A(G,d)_n\right)^{G^{n-a}}\) as \(S_n\)-representations. Next, since \(A(G,d)_n\) is spanned by the \(G^n\) translates of \(A(d)_n\), we know that \(\left(A(G,d)_n\right)^{G^{n-a}}\) will be spanned by \(G^a\) translates of \(A(d)_n\). In conclusion,
\begin{equation} \label{convg_calc}
\dim \left(H^*(X^n; \mathbb{Q}) \otimes A(G,d)_n\right)^{W_{n-a}} \le |G|^a \dim \left(H^*(X^n; \mathbb{Q}) \otimes A(d)_n\right)^{S_{n-a}}
\end{equation}
But the right-hand side of (\ref{convg_calc}) is just (a constant times) exactly the term that Farb-Wolfson considered in \cite[Thm 3.4]{FW2} and proved was convergent. Therefore \(H^*(\Conf^G(X); \mathbb{Q})\) is convergent.
\end{proof}

\subsection{Polynomial bounds for branched covers of \(\mathbb{P}^1\)}
Consider the specific case where \(X\) is a surface with punctures and \(X/G\) is a sphere with punctures. For example, take any holomorphic branched cover of the Riemann sphere and remove the branch points. In this case, we can do better than proving that \(H^i(\Conf^G(X))^{W_{n-a}}\) is bounded by a subexponential function in \(i\): we can prove it is actually bounded by a polynomial in \(i\). Moreover, our proof is a straightforward topological one, and avoids relying on the detailed calculations of Farb-Wolfson. This also provides a partial answer to a question of Chen \cite{WCh2}, who proved that for \(X = \mathbb{C}\), \(\langle H^i(\Conf(X)), P\rangle\) is actually given by a quasipolynomial in \(i\), and asked for a topological proof of this fact. We believe that such quasipolynomiality probably exists for a more general class of manifolds. We hope all this provides some motivation for our reproving a special case of \thref{conf_convg} by other means.

So take \(X = \mathcal{S} - \{p_1, \dots, p_b\}\), with a free action of the finite group \(G\), and \(X/G = S^2 - \{q_0, \dots, q_c\} \approx \mathbb{C} - \{q_1, \dots, q_c\}\), where we assume there is at least one branch point, so \(c \ge 0\). To begin, notice that by transfer,
\[H^i(\Conf^G_n(X); \mathbb{Q})^{W_{n-a}} \cong H^i(\Conf^G_n(X)/W_{n-a}; \mathbb{Q}) \]
For any fixed \(a < n\), put \(Y_{n,a} = \Conf^G_n(X)/W_{n-a}\), so that
\[ Y_{n,a} =  \left\{\{x_1, \dots, x_{n-a}\}, (z_1, \dots, z_a) \in \UConf_{n-a}(X/G) \times \Conf^G_a(X) \mid x_i \ne \pi(z_j)\right\}\]
There is thus a fiber bundle
\begin{equation} \label{convg_fibr}
\UConf_{n-a}(X/G - \{\pi(z_1), \dots \pi(z_a)\}) \hookrightarrow Y_{n,a} \twoheadrightarrow \Conf^G_a(X)
\end{equation}
and we have the following.
\begin{prop} \thlabel{monodromy}
In the fiber bundle (\ref{convg_fibr}), the monodromy action
\[\pi_1(\Conf^G_a(X)) \curvearrowright H^*(\UConf_{n-a}(X/G - \{\pi(z_1), \dots \pi(z_a)\}) \]
is trivial.
\end{prop}
\begin{proof}
There is an inclusion
\begin{align*}
\pi_1(\Conf^G_a(X), (z_1, \dots, z_a)) &\subset \pi_1(\UConf_a(X/G), \{\pi(z_1), \dots, \pi(z_a)\}) \\
&= \pi_1(\UConf_a(\mathbb{C} - \{q_1, \dots, q_c\}), \{\pi(z_1), \dots, \pi(z_a)\}) \\
&=  \text{Mod}(\mathbb{C} - \{q_1, \dots, q_c\}, \{\pi(z_1), \dots, \pi(z_a)\})
\end{align*}
into the mapping class group with marked points, via the well-known isomorphism between braid groups and mapping class groups. That is, the group of homeomorphisms up to isotopy of \(\mathbb{C}\) that fix \(\{q_1, \dots, q_c\}\) pointwise and leave \(\{\pi(z_1), \dots, \pi(z_a)\}\) invariant, but possibly permute them. Thus \(\pi_1(\Conf^G_a(\mathbb{C}^*))\) acts (up to homotopy and conjugacy) by ``point-pushing'' diffeomorphisms on the underlying space \(X/G - \{\pi(z_1), \dots, \pi(z_a)\}\), and the action of \(\pi_1(\Conf^G_a(X))\) on the fiber of (\ref{convg_fibr}) is just the diagonal action.

Now, \(\UConf_{n-a}(\mathbb{C} - \{q_1, \dots, q_c, \pi(z_1), \dots \pi(z_a)\})\) has a normal \(S_{n-a}\)-cover
\[\Conf_{n-a}(\mathbb{C} - \{q_1, \dots, q_c, \pi(z_1), \dots \pi(z_a)\}) \to \UConf_{n-a}(\mathbb{C} - \{q_1, \dots, q_c, \pi(z_1), \dots \pi(z_a)\}) \]
which is a hyperplane arrangement. By transfer,
\begin{gather*}
H^*(\UConf_{n-a}(\mathbb{C} - \{q_1, \dots, q_c, \pi(z_1), \dots \pi(z_a)\}); \mathbb{Q}) \\
\cong \left(H^*(\Conf_{n-a}(\mathbb{C} - \{q_1, \dots, q_c, \pi(z_1), \dots \pi(z_a)\}); \mathbb{Q})\right)^{S_{n-a}}
\end{gather*}
and by Orlik-Solomon \cite{OS}, \(H^*(\Conf_{n-a}(\mathbb{C} - \{q_1, \dots, q_c, z_1^d, \dots z_a^d\}); \mathbb{Q})\) is generated as an algebra by \(H^1\). So it is enough to determine the action of \(\pi_1(\UConf_a(\mathbb{C} - \{q_1, \dots, q_c\}))\) on \(H^1(\Conf_{n-a}(\mathbb{C} - \{q_1, \dots, q_c, \pi(z_1), \dots \pi(z_a)\}); \mathbb{Q})\). But \(H^1\) is spanned by the de Rham cohomology classes
\begin{equation} \label{dR}
\left\{\frac{d(w_i - w_j)}{w_i - w_j} \,\middle|\, 1 \le i < j \le n-a\right\}, \left\{\frac{dw_i}{w_i - q_k} \,\middle|\, \substack{1 \le i \le n-a \\ 1 \le k \le c}\right\}, \left\{\frac{dw_i}{w_i - \pi(z_l)} \,\middle|\, \substack{1 \le i \le n-a \\ 1 \le l \le a}\right\}
\end{equation}
By tracking how points are pushed by a mapping class, we see that a braid in \(\pi_1(\UConf_a(\mathbb{C} - \{q_1, \dots, q_c\}))\) acts on the first two sets of cohomology classes in (\ref{dR}) trivially, and on the third set in (\ref{dR}) by permuting the \(\{\pi(z_1), \dots \pi(z_a)\}\) under the homomorphism \(\pi_1(\UConf_a(\mathbb{C} - \{q_1, \dots, q_c\})) \to S_a\). But \(\pi_1(\Conf_a(X/G))\) is precisely the kernel of this homomorphism, the subgroup of \emph{pure} braids. So the action of \(\pi_1(\Conf^G_a(X))\) on \(H^1(\Conf_{n-a}(\mathbb{C} - \{q_1, \dots, q_c, z_1^d, \dots z_a^d\}))\), and therefore as we have argued, on \(H^*(\UConf_{n-a}(\mathbb{C} - \{q_1, \dots, q_c, \pi(z_1), \dots \pi(z_a)\}); \mathbb{Q})\), is trivial.
\end{proof}

\begin{thm}
Let \(X\) be a surface with punctures, with a free action of a finite group \(G\), such that \(X/G\) is a sphere with punctures. Then for fixed \(a\), the function
\[i \mapsto \dim H^i(\Conf^G_n(\mathbb{C}^*); \mathbb{Q})^{W_{n-a}}\]
is bounded by a polynomial in \(i\) and uniformly bounded in \(n\).
\end{thm}
This is a big improvement on \cite{CEF2}, where they were only able to bound the dimension by \(O(e^{\sqrt{i}})\).
\begin{proof}

Because the action is trivial by \thref{monodromy}, in the Serre spectral sequence for the bundle (\ref{convg_fibr}), the coefficients are constant. Thus the \(E_2\) page of this spectral sequence is just
\begin{align*}
E_2^{p,q} &= H^p(\Conf^G_a(X); H^q(\UConf_{n-a}(X/G - \{\pi(z_1), \dots \pi(z_a)\}); \mathbb{Q})) \\
&= H^p(\Conf^G_a(X)) \otimes H^q(\UConf_{n-a}(X/G - \{\pi(z_1), \dots \pi(z_a)\}); \mathbb{Q}) \implies H^{p+q}(Y_{n,a}; \mathbb{Q})
\end{align*}
Abutment of spectral sequences means that \(H^{i}(Y_{n,a}; \mathbb{Q})\) is some subquotient of the direct sum of the diagonal terms
\[\bigoplus_{p + q = i} H^p(\Conf^G_a(X); \mathbb{Q}) \otimes H^q(\UConf_{n-a}(X/G - \{\pi(z_1), \dots \pi(z_a)\}); \mathbb{Q})\]
Since taking subquotients only decreasing the vector space dimension, we obtain
\begin{align*} \dim H^{i}(Y_{n,a}; \mathbb{Q}) &\le \sum_{p+q = i} \dim H^p(\Conf^G_a(X); \mathbb{Q}) \cdot \dim H^q(\UConf_{n-a}(X/G - \{\pi(z_1), \dots \pi(z_a)\}); \mathbb{Q}) \\
&= O\left(\dim H^i(\UConf_{n-a}(X/G - \{\pi(z_1), \dots \pi(z_a)\}); \mathbb{Q})\right)
\end{align*}
because \(H^p(\Conf^G_a(X); \mathbb{Q}) = 0\) for \(p > 2a\), since it is a \(2a\)-dimensional manifold. So it is enough to show that
\[ \dim H^i(\UConf_{n-a}(\mathbb{C} - \{q_1, \dots, q_c, \pi(z_1), \dots \pi(z_a)\}); \mathbb{Q}) \]
is bounded by a polynomial in \(i\) and is uniformly bounded in \(n\). We can simplify this description by redefining \(n := n - a\) and \(a := c + a\), and picking any \(a\) points, so we are left to analyze \(\dim H^i(\UConf_n(\mathbb{C} \setminus [a]))\), where \([a] = \{1, \dots, a\}\).

Notice that \(\mathbb{C} \setminus [a]\) is a connected orientable manifold of dimension 2 with finite-dimensional homology, so it satisfies the hypotheses of homological stability for unordered configuration spaces \cite[Thm 1]{Chu}. Therefore \(\dim H^i(\UConf_n(\mathbb{C} \setminus [a]); \mathbb{Q})\) is independent of \(n\) for \(n \gg 0\), and thus uniformly bounded in \(n\). We claim that in fact
\[\dim H^i(\UConf_n(\mathbb{C} \setminus [a]); \mathbb{Q}) = O(i^{a-1})\]
Indeed, notice that we have a decomposition:
\begin{equation} \label{decomp}
\UConf_n(\mathbb{C} \setminus [a-1]) = \UConf_n(\mathbb{C} \setminus [a]) \sqcup \UConf_{n-1}(\mathbb{C} \setminus [a])
\end{equation}
This decomposition just says that an unordered collection of \(n\) distinct points in \(\mathbb{C} \setminus [a-1]\) either does not contain the \(a\)-th point, or it does and so is determined by the \(n-1\) other points. \(\UConf_n(\mathbb{C} \setminus [a])\) is an open complex submanifold of \(\UConf_n(\mathbb{C} \setminus [a-1])\), and so there is a Thom-Gysin exact sequence associated to (\ref{decomp}):
\begin{gather} \label{les} \begin{aligned}
\dots &\to H^{i-2}(\UConf_{n-1}(\mathbb{C} \setminus [a]); \mathbb{Q}) \to H^i(\UConf_n(\mathbb{C} \setminus [a-1]); \mathbb{Q}) \to H^i(\UConf_n(\mathbb{C} \setminus [a]); \mathbb{Q}) \\
&\to H^{i-1}(\UConf_{n-1}(\mathbb{C} \setminus [a]); \mathbb{Q}) \to \dots
\end{aligned} \end{gather}
By exactness of (\ref{les}):
\begin{equation} \label{ex_bd}
\dim H^i(\UConf_n(\mathbb{C} \setminus [a]); \mathbb{Q}) \le \dim H^i(\UConf_n(\mathbb{C} \setminus [a-1]); \mathbb{Q}) + \dim H^{i-1}(\UConf_{n-1}(\mathbb{C} \setminus [a]); \mathbb{Q}) \end{equation}
By repeatedly plugging in the third term of (\ref{ex_bd}) back into the left-hand side of (\ref{ex_bd}), we obtain
\begin{gather} \label{ind_bd} \begin{aligned}
\dim H^i(\UConf_n(\mathbb{C} \setminus [a]); \mathbb{Q}) \le &\dim H^i(\UConf_n(\mathbb{C} \setminus [a-1]); \mathbb{Q}) + \dim H^{i-1}(\UConf_{n-1}(\mathbb{C} \setminus [a-1]); \mathbb{Q}) \\
&+ \dots + \dim H^1(\UConf_1 (\mathbb{C} \setminus [a-1]); \mathbb{Q}) \\
&= O\left(i \cdot \dim H^i(\UConf_n(\mathbb{C} \setminus [a-1]); \mathbb{Q})\right)
\end{aligned} \end{gather}
Now we induct on \(a\). Notice that for \(a = 0\), we know \(H^i(\UConf_n(\mathbb{C}); \mathbb{Q}) = 0\) once \(i \ge 2\). Therefore for \(a = 1\), this implies that \(\dim H^i(\UConf_n(\mathbb{C}^*); \mathbb{Q}) = O(1)\). By induction, we can assume that \(\dim H^i(\UConf_n(\mathbb{C} \setminus [a-1]); \mathbb{Q}) =  O(i^{a-2})\). The bound (\ref{ind_bd}) then tells us that
\[ \dim H^i(\UConf_n(\mathbb{C} \setminus [a]); \mathbb{Q}) = O\left(i \cdot \dim H^i(\UConf_n(\mathbb{C} \setminus [a-1]); \mathbb{Q})\right) = O(i \cdot i^{a-2}) = O(i^{a-1}) \]
as desired.
\end{proof}

\section{Arithmetic statistics}
Let \(Z\) be a smooth quasiprojective scheme over \(\mathcal{O}_k[1/N]\). Suppose that \(W_n = G^n \rtimes S_n\) acts freely on \(Z\) by automorphisms, and let \(Y = Z/W_n\) be the quotient, which is known to be a scheme.

For any prime \(\mathfrak{p}\) of \(\mathcal{O}_k[1/N]\), we have \(\mathcal{O}_k/\mathfrak{p} \cong \mathbb{F}_q\), where \(q\) is a power of the rational prime \(p\) under \(\mathfrak{p}\). We can then base-change \(Y\) to \(\overline{\mathbb{F}}_q\). The geometric Frobenius \(\Frob_q\) then acts on \(Y_{/\overline{\mathbb{F}}_q}\). The fixed-point set of \(\Frob_q\) is exactly \(Y(\mathbb{F}_q)\).

Fix a prime \(l \ne p\), and let \(L\) be a splitting field for \(G\) over \(\mathbb{Q}_l\). Since all the irreducible representations of \(G\) are defined over \(L\), there is a natural correspondence between finite-dimensional representations of \(W_n\) over \(L\) and finite-dimensional constructible \(L\)-sheaves on \(Y\) that become trivial when pulled back to \(Z\).

Given a representation \(V\) of \(W_n\), let \(\chi_V\) be its character and \(\mathcal{V}\) the associated sheaf on \(Y\). For any point \(y \in Y(\mathbb{F}_q)\), since \(Frob_q\) fixes \(y\), then \(\Frob_q\) acts on the stalk \(\mathcal{V}_y\), which is isomorphic to \(V\). This action determines an element \(\sigma_y \in W_n\) up to conjugacy, so that \(\tr(\Frob_q : \mathcal{V}_y) = \chi_V(\sigma_y)\). The Grothendieck-Lefschetz trace formula says that
\[ \sum_{y \in Y(\mathbb{F}_q)} \tr(\Frob_q: \mathcal{V}_y) = \sum_i (-1)^i \tr\left(\Frob_q : H^i_{\acute et, c}(Y_{/\overline{\mathbb{F}}_q}; \mathcal{V})\right) \]

We then have the following chain of equalities, as in \cite{CEF2} and \cite{FW2}:
\begin{align*}
H^i_{\acute et, c}(Y_{/\overline{\mathbb{F}}_q}; \mathcal{V}) &\cong (H^i_{\acute et, c}(Z; \pi^* \mathcal{V}))^{W_n} & \text{by transfer}\\
&\cong \left(H^i_{\acute et, c}(Z; L) \otimes V\right)^{W_n} & \text{by triviality of pullback} \\
&\cong \left(H^{2\dim Z - i}_{\acute et}(Z; L(-\dim Z))^* \otimes V\right)^{W_n} &\text{by Poincar\'e duality} \\
&\cong \langle H^{2\dim Z - i}_{\acute et}(Z; L(-\dim Z)), V \rangle_{W_n}
\end{align*}
and so we obtain
\begin{equation}
\label{single_grot}
\sum_{y \in Y(\mathbb{F}_q)} \chi_V(\sigma_y) = q^{\dim Z} \sum_i (-1)^i \tr\left( \Frob_q : \langle H^i_{\acute et}(Z; L), V \rangle_{W_n}\right)
\end{equation}

We would like to apply (\ref{single_grot}) to a sequence of schemes \(Z_n\) that form a co-\(\FI_G\)-scheme, and then let \(n \to \infty\). To make this work, we need to know that \(Z\) satisfies \'etale representation stability, and that \(H^*(Z)\) is convergent in the sense of \S3. Following \cite{FW2}, we have the following.

\begin{thm}
\thlabel{fig_stats}
Suppose that \(Z\) is a smooth quasiprojective co-\(\FI_G\)-scheme over \(\mathcal{O}_k[1/N]\) such that \(H^i(Z_{/\overline{\mathbb{F}}_q}; \mathbb{Q}_l)\) is a finitely-generated \(\Gal(\overline{\mathbb{F}}_q/\mathbb{F}_q)\)-\(\FI_G\)-module, and that \(H^*(Z; \mathbb{Q}_l)\) is convergent. Then for any \(\FI_G\) character polynomial \(P\),
\begin{equation}
\lim_{n \to \infty} q^{-\dim Z_n} \sum_{y \in Y_n(\mathbb{F}_q)} P(\sigma_y) = \sum_{i=0}^\infty (-1)^i \tr\left( \Frob_q : \langle H^i_{\acute et}(Z; L), P \rangle\right)
\end{equation}
\end{thm}
\begin{proof}
Since each \(Z_n\) is smooth quasiprojective, we can apply (\ref{single_grot}) to it. By linearity, we can replace a representation \(V\) of \(W_n\) with a virtual representation given by a character polynomial \(P\), so we obtain
\begin{equation}
q^{-\dim Z_n} \sum_{y \in Y_n(\mathbb{F}_q)} P(\sigma_y) = \sum_{i=0}^{2\dim Z_n} (-1)^i \tr\left( \Frob_q : \langle H^i_{\acute et}(Z_n; L), P \rangle_{W_n}\right)
\end{equation}
Call this sum \(A_n\). Furthermore, let
\[ B_n = \sum_{i=0}^{2\dim Z_n} (-1)^i \tr\left( \Frob_q : \langle H^i_{\acute et}(Z; L), P \rangle\right) \]
Our goal is thus to show that \(\lim_{n \to \infty} A_n = \lim_{n \to \infty} B_n\): that is, first of all that both sequences converge, and that their limits are equal.

By our assumption that \(H^*(Z)\) is convergent, we know that there is a function \(F_P(i)\) which is subexponential in \(i\), such that for all \(n\),
\[ |\langle H^i_{\acute et}(Z_n; L), P\rangle_{W_n}| \le F_P(i) \]
and thus in particular (taking \(n\) large enough)
\[ |\langle H^i_{\acute et}(Z; L), P\rangle| \le F_P(i) \]
Furthermore, by Deligne \cite[Thm 1.6]{De}, we know that
\[ \left|\tr\left( \Frob_q : \langle H^i_{\acute et}(Z_n; L), P \rangle\right)\right| \le q^{-i/2} \left|\langle H^i_{\acute et}(Z_n; L), P \rangle\right| \]
We thus have
\begin{align*}
|A_n| &\le \sum_{i=0}^{2\dim Z_n} \left|\tr\left( \Frob_q : \langle H^i_{\acute et}(Z_n; L), P \rangle_{W_n}\right)\right| \\
&\le \sum_{i=0}^{2\dim Z_n} q^{-i/2} \left|\langle H^i_{\acute et}(Z_n; L), P \rangle\right| \\
&\le \sum_{i=0}^{2\dim Z_n} q^{-i/2} F_P(i)
\end{align*}
For exactly the same reason, \(|B_n| \le \sum_{i=0}^{2\dim Z_n} q^{-i/2} F_P(i)\). Since \(F_P(i)\) is subexponential in \(i\), this means that both \(A_n\) and \(B_n\) converge. 

It remains to show that \(\lim_{n \to \infty} A_n - B_n = 0\). Let \(N(n,P)\) be the number such that
\[\langle H^i(Z_n), P\rangle_{W_n} = \langle H^i(Z), P\rangle \;\; \text{ for all } i \le N(n,P). \]
We thus have
\begin{align*}
|B_n - A_n| &\le \sum_{i=0}^{2\dim Z_n} q^{-i/2} \left| \langle H^i(Z), P\rangle - \langle H^i(Z_n), P\rangle_{W_n} \right| \\
&= \sum_{i=N(n,P)+1}^{2\dim Z_n} q^{-i/2} \left| \langle H^i(Z), P\rangle - \langle H^i(Z_n), P\rangle_{W_n} \right| \\
&= \sum_{i=N(n,P)+1}^{2\dim Z_n} 2q^{-i/2} F_P(i)
\end{align*}
Since \(N(n,P) \to \infty\) as \(n \to \infty\), and \(F_P(i)\) is subexponential in \(i\), we conclude that \(|B_n - A_n|\) becomes arbitrarily small as \(n \to \infty\).
\end{proof}

We can now apply this to \(\Conf^G_n(X)\) to obtain \thref{arith_stats}.
\begin{proof}[Proof of \thref{arith_stats}]
By \thref{etale_repstab}, \(H^i(\Conf^G(X)_{/\overline{\mathbb{F}}_q}\) is a finitely-generated \(\Gal(\overline{\mathbb{F}}_q/\mathbb{F}_q)\)-\(\FI_G\)-module. By \thref{conf_convg} \(H^*(\Conf^G(X))\) is convergent. We thus conclude by \thref{fig_stats}.
\end{proof}

\section{Point-counts for polynomials}
For special choices of \(X\) in \thref{arith_stats}, we can give an interpretation to the left-hand side of (\ref{groth}) in terms of point-counts of polynomials over \(\mathbb{F}_q\). 

\subsection{Complex reflection groups and Gauss sums}
The first example is where \(X = \mathbb{G}_m\) and \(G = \mathbb{Z}/d\mathbb{Z}\) acting by multiplication by a \(d\)-th root of unity. In order for this action to be well-defined, we need to consider \(X\) as a scheme over \(\mathbb{Z}[\zeta_d]\). However, notice that here the action is not free: if we look at the fiber of \(X\) over a prime dividing \(d\), then \(\mathbb{Z}/d\mathbb{Z}\) will act trivially, because these are the primes that ramify in \(\mathbb{Z}[\zeta_d]\). Thus, to satisfy our hypotheses that we have a free action of \(G\), we consider \(X\) as a scheme over \(\mathcal{O}_k[1/d]\), where \(k = \mathbb{Q}(\zeta)\) is the cyclotomic field. Thus the finite fields \(\mathbb{F}_q\) we consider will satisfy \(q \equiv 1 \mod d\), since \(\mathbb{F}_q\) is a residue field of \(\mathbb{Z}[\zeta]\).

In this case, we have
\[ \Conf^{\mathbb{Z}/d\mathbb{Z}}_n(\mathbb{G}_m)(R) = \{(x_i) \in R^n \mid x_i \ne 0, x_i \ne \zeta^k x_j\} \]
so that \(\Conf^G_n(X)\) is the complement of a hyperplane arrangement. The arithmetic and \'etale cohomology of this arrangement was studied by Kisin-Lehrer in \cite{KL}, where they obtained formulas for the equivariant Poincar\'e polynomial of \(\Conf^{\mathbb{Z}/d\mathbb{Z}}_n(\mathbb{G}_m)\).

By Bj\"orner-Ekedahl \cite{BE}, the action of \(\Frob_q\) on \(H^i_{\acute et, c}(Z; \mathbb{Q}_l)\) is given by multiplication by \(q^i\), and thus (again, by Poincar\'e duality)
\[\tr\left(\Frob_q : \langle H^i_{\acute et}(\Conf^G_n(X)), P\rangle\right) = q^{-i} \langle H^i_{\acute et}(\Conf^G_n(X)), P\rangle \]
This lets us compute the right-hand side of (\ref{groth}) explicitly in this case:
\begin{equation} \label{crg_groth}
\lim_{n \to \infty} q^{-n} \sum_{f \in \Poly_n(\mathbb{F}_q^*)} P(y) = \sum_{i=0}^\infty (-1)^i q^{-i} \langle H^i(\Conf^{\mathbb{Z}/d\mathbb{Z}}(\mathbb{C}^*); \mathbb{C}), P \rangle
\end{equation}
since we know by \thref{etale_repstab} that \(H^i_{\acute et}(\Conf^{\mathbb{Z}/d\mathbb{Z}}(\mathbb{G}_m)_{/\overline{\mathbb{F}}_q}; \mathbb{Q}_l) \cong H^i(\Conf^{\mathbb{Z}/d\mathbb{Z}}(\mathbb{C}^*); \mathbb{Q}_l)\).

Notice that \(d \mid q - 1\), so \(\mathbb{Z}/d\mathbb{Z}\) is a quotient of \(\mathbb{Z}/(q-1)\mathbb{Z}\). Thus any representation of \(\mathbb{Z}/d\mathbb{Z} \wr S_n\) lifts to a representation of \(\mathbb{Z}/(q-1)\mathbb{Z} \wr S_n\), so we can always interpret the left-hand side of (\ref{crg_groth}) as a statement about representations of the single group \(\mathbb{Z}/(q-1)\mathbb{Z} \wr S_n\). On the other hand, as remarked in \cite[\S3]{Ca1}, there is a Galois cover
\[\Conf^{\mathbb{Z}/(q-1)\mathbb{Z}}_n(\mathbb{C}^*) \to \Conf^{\mathbb{Z}/d\mathbb{Z}}_n(\mathbb{C}^*)\]
with deck group \((\mathbb{Z}/\frac{q-1}{r}\mathbb{Z})^n\) and so by transfer,
\[ H^i(\Conf^{\mathbb{Z}/d\mathbb{Z}}_n(\mathbb{C}^*); \mathbb{Q}) = \left(H^i(\Conf^{\mathbb{Z}/(q-1)\mathbb{Z}}_n(\mathbb{C}^*); \mathbb{Q})\right)^{(\mathbb{Z}/\frac{q-1}{r}\mathbb{Z})^n} \]
and thus the right-hand side of (\ref{crg_groth}) is the same whether we consider \(V\) as a representation of \(\mathbb{Z}/d\mathbb{Z} \wr S_n\) or \(\mathbb{Z}/(q-1)\mathbb{Z} \wr S_n\). Therefore we lose nothing if we simply assume that in fact \(d = q-1\).

Now we give some number-theoretic meaning to the left-hand side of (\ref{crg_groth}). An element \(f \in \Poly_n(\mathbb{F}_q^*)\) is a polynomial in \(\mathbb{F}_q[T]\) that does not have 0 as a root. The roots of \(f(T)\) are sitting in some extension field of \(\mathbb{F}_q\), and the \(d\)-th roots of those roots possibly in some even higher extension field. The permutation \(\sigma_f\) is the action of \(\Frob_q\) on all these \(d\)-th roots, which permutes the actual roots (think of these as the columns each containing a set of \(d\)-th roots), and then further permutes the \(d\)-th roots cyclically. This precisely gives an element of \(G^n \rtimes S_n\) (up to conjugacy).

Recall that a \emph{Gauss sum} is a certain sum of roots of unity obtained by summing values of a character of the unit group of a finite ring. Now, suppose that \(\chi\) is an irreducible character of \(G\). Write
\[X^\chi_i = \sum_{g \in G} \chi(g) X^g_i.\]
For each \(i\), the \(\{X^\chi_i\}_{\chi \in \widehat{G}}\) have the same span as \(\{X^g_i\}_{g \in G}\), and the \(\{X^\chi_i\}\) are more natural to use here. 

Since \(d = q - 1\), there is an isomorphism \(\mathbb{F}^*_q \cong G\), and in fact such a surjection \(\mathbb{F}_{q^k}^* \to G\) for any \(k\). It therefore makes sense to talk about applying \(\chi\) to elements of \(\overline{\mathbb{F}}_q^*\). For a given \(f \in \Poly_n(\mathbb{F}^*_q)\), consider first decomposing \(f\) into irreducibles factors over \(\mathbb{F}_q\), and then each of these into linear factors over \(\overline{\mathbb{F}}_q\); since none are zero, all the roots actually lie in \(\overline{\mathbb{F}}_q^*\). For a given irreducible degree-\(i\) factor \(p\) of \(f\), since \(\Frob_q\) acts transitively on the roots of \(p\) over \(\overline{\mathbb{F}}_q\), and since \(q \equiv 1 \mod d\), then \(\chi\) takes the same value on each of the roots. It is then straightforward to calculate that
\[
X^\chi_i(\sigma_f) = \sum_{\deg(p) = i} \chi(\text{root}(p))
\]
where the sum is taken over all irreducible factors \(p\) of \(f\) of degree \(i\), and \(\text{root}(p)\) denotes any of the roots of \(p\) in \(\overline{\mathbb{F}}_q\). Thus, \(X^\chi_i\) is a Gauss sum of \(\chi\) applied to the roots of degree-\(i\) irreducible factors of \(f\). A general character polynomial is generated as a ring by the \(X^\chi_i\)'s, so this says how to interpret the left-hand-side. Thus, (\ref{crg_groth}) says that the \emph{average value} of any such Gauss sum across all polynomials in \(\Poly_n(\mathbb{F}^*_q)\) always converges to the series in \(q^{-1}\) on the right-hand side of (\ref{crg_groth}). In particular, the decomposition of \(H^1\) and \(H^2\) determined in \cite[\S4.1]{Ca1} allows us to compute the examples (\ref{X1_pcount}) and (\ref{X2_pcount}) from the introduction.

\subsection{Counting solutions to \(f = g^d - t h^d\)}
Consider counting the number of monic polynomials \(f \in \mathbb{F}_q[t]\) with distinct roots of the form \(f = g^d - t h^d\) for \(g, h \in \mathbb{F}_q[t]\). If we assume \(q \equiv 1\) mod d, so that there is a primitive \(d\)-th root of unit \(\zeta \in \mathbb{F}_q\), then we can factor
\begin{align*}
f(t^d) = g^d(t^d) - t^d b^d(t^d) = (a(t^d) - t b(t^d)) (a(t^d) - \zeta t b(t^d)) \cdots (a(t^d) - \zeta^{d-1} b(t^d))
\end{align*}
where each \(a(t^d) - \zeta^i t b(t^d) \in \mathbb{F}_q[t]\). Notice that if \(x\) is a root of \(a(t^d) - t b(t^d)\), then \(\zeta^{n-i} x\) is a root of \(a(t^d) - \zeta^i t b(t^d)\). Thus, \(f(t^d)\) has such a factorization if and only if the \(d\)-th roots of the roots of \(f\) all lie in different Galois orbits. This is therefore equivalent to being able to write \(f = g^d + t h^d\). Also notice that this is equivalent to \(f\) being a norm in the cyclic extension \(\mathbb{F}[\sqrt[d] t] / \mathbb{F}[t]\).

Now, assume further that \(f\) is monic with distinct roots, and that 0 is not a root of \(f\), so that \(f \in \Poly_n(\mathbb{F}_q^*)\). Then as we have seen, we get a group element \(\sigma_f \in W_n\), up to conjugacy. The condition that the \(d\)-th roots of the roots of \(f\) all lie in different Galois orbits precisely says that \(\sigma_f\) is in a conjugacy class where each cycle is decorated with the identity element of \(G\). Let \(\delta_n\) be the indicator function for such conjugacy classes. Then as we have just argued, \(\delta_n(\sigma_f) = 1\) precisely when we can write \(f = g^d + t h^d\).

Then (\ref{single_grot}) tells us that
\begin{gather*}\#\{f \in \Poly_n(\mathbb{F}_q^*) \mid f = g^d + t h^d \text{ for } g,h \in \mathbb{F}_q[t]\} = \sum_{f \in \Poly_n(\mathbb{F}_q^*)} \delta_n(\sigma_f) \\
= q^n \sum_{i = 0}^{2n} (-1)^i q^{-i} \langle H^i(\Conf^{\mathbb{Z}/d\mathbb{Z}}(\mathbb{C}^*)), \delta_n\rangle
\end{gather*}
This formula was investigated in great detail in the case \(d = 2\) by Matei \cite{Mat}, in which case the equation \(f = g^2 - t h^2\) had been studied in a number of previous papers, e.g. \cite{BSSW}.

It is very important to note that while each \(\delta_n\) is a class function of \(W_n\), and thus in the span of character polynomials for \(W_n\), there is \emph{no} character polynomial \(P\) such that \(P_n = \delta_n\). This can be seen, for instance, by the fact that a single character polynomial can only depend on cycles of bounded length, while \(\delta_n\) depends on all the cycles of \(g \in W_n\). Thus we cannot apply \thref{arith_stats} and take the limit as \(n \to \infty\).

Indeed, the calculations in \cite{Mat} make it clear that the result depends on \(n\), although there is a certain kind of stability present in his results. It would be very interesting to carry out the same calculations for general \(d\).

\subsection{Iterated configuration spaces}
Next, take \(X = \Conf_d(\mathbb{A}^1)\) and \(G = S_d\) permuting the indices, which by definition of \(\Conf_d\) is a free action. Thus we are considering \(\Conf^{S_d}_n(\Conf_d(\mathbb{A}^1))\) with the action of \(W_n = S_d \wr S_n\), and quotient \(\UConf_n(\UConf_d(\mathbb{A}^1))\). Here we have
\begin{align*} \Conf^{S_d}_n(\Conf_d(\mathbb{A}^1))(R) = \big\{&[ (x^1_1, \dots, x^1_d), \dots, (x^n_1, \dots, x^n_d)] \in (R^d)^n : \\
& x^i_a \ne x^i_b,  \{x^i_1, \dots, x^i_d\} \ne \{x^j_1, \dots, x^j_d\} \big\}
\end{align*}
Thus, \(\Conf^G_n(X)\) is the complement of a linear subspace arrangement: the subspaces are the codimension 1 spaces \(\{x^i_a = x^i_b\}\), and the codimension \(d\) spaces \(\{(x^i_1, \dots, x^i_d) = (x^j_1, \dots, x^j_d)\}\) and all its \(S_d\)-translates. Unfortunately, there are not simple results on the eigenvalues of Frobenius for a linear subspace arrangement, like there are for hyperplanes arrangements, because the different codimensions of the subspaces ``mix up'' the action of Frobenius on different cohomology classes. So all we are able to assert is the content of \thref{arith_stats} for this specific \(X\) and \(G\).
\begin{equation}
\label{iter_groth}
\lim_{n \to \infty} q^{-nd} \sum_{y \in \UConf_n(\UConf_d(\mathbb{F}_q))} P(\sigma_y) = \sum_{i = 0}^\infty (-1)^i \tr\left(\Frob_q : \langle H^i(\Conf^{G}(X); L), P\rangle\right)
\end{equation}
However, we are able to interpret the left-hand side of (8). An element \(\{f_i\} \in \UConf_n(\UConf_d(\mathbb{F}_q))\) is a set \(\{f_1, \dots, f_n\}\) of distinct monic polynomials \(f_i \in \overline{\mathbb{F}}_q[t]\) with \(d\) distinct roots such that \(\Frob_q(f_i) = f_j\). In particular, the product \(f_1 f_2 \cdots f_n \in \mathbb{F}_q[t]\). The element \(\sigma_{\{f_i\}}\) is the action of Frobenius on all the roots of the \(\{f_i\}\), which permute the \(\{f_i\}\) themselves and then further permute each of their roots. This precisely gives an element of \(S_d \wr S_n\), up to conjugacy.

A character polynomial \(P\) for \(S_d \wr S_n\) is a polynomial in \(\{X^\lambda_i \mid 1 \le i \le n, \lambda \vdash d\}\), where \(X^\lambda_i\) counts the number of \(\lambda\)-decorated \(i\)-cycles. For us, this means that we look at the cycle decomposition of how \(\Frob_q\) acts on the \(\{f_j\}\), and on the \(d\) roots of each \(f_j\). This will give us a cycle decomposition of \(n\), the action (up to conjugacy) on the \(\{f_j\}\), and each of the cycles will be labeled by a partition (cycle decomposition) of \(d\), which is the action on the \(d\) roots. \(X_i^{\lambda}\) then counts the number of \(i\)-cycles labeled with the partition \(\lambda\).

That is to say, \(X_i^{\lambda}\) counts the number of \(\{f_{j_1}, \dots, f_{j_i}\} \subset \{f_j\}\) such that Frobenius permutes the \(\{f_{j_1}, \dots, f_{j_i}\}\) cyclically, and that the product \(f_{j_1} \cdots f_{j_i}\), which must therefore lie in \(\mathbb{F}_q[t]\), decompose into irreducible pieces of degrees \(i \lambda_1, \dots, i \lambda_m\).

For example, if \(\lambda\) is a \(d\)-cycle, then \(X_2^{\lambda}\) counts the number of \(\{f_\alpha, f_\beta\} \in \{f_j\}\) such that \(f_\alpha f_\beta\) is an irreducible polynomial in \(\mathbb{F}_q[t]\). We can think of the sum
\[ \sum_{\{f_j\} \in \UConf_n(\UConf_d(\mathbb{F}_q))} X_2^{\lambda}(\sigma_{\{f_j\}})\]
as being over all monic \(F \in \mathbb{F}_q[t]\) with distinct roots of degree \(n \cdot d\), where we sum over all the possible ways of decomposing \(F = f_1 \cdots f_n\). This sum then counts the number of ways of decomposing \(F = G \cdot f_1 \cdots f_{n-2}\), where \(G \in \mathbb{F}_q[t]\) is irreducible of degree \(2d\), and \(f_j \in \overline{\mathbb{F}}_q[t]\), each of degree \(d\), satisfy \(\Frob(f_i) = f_j\). Thus, (\ref{iter_groth}) is saying that the \emph{average} number of such decompositions, across all such \(F\) of degree \(n \cdot d\), stabilizes to the right-hand side of (\ref{iter_groth}) as \(n \to \infty\).

\bibliography{general}

\begin{thebibliography}{BSSW16}

\bibitem[BE97]{BE}
A.~Bj{\"o}rner and T.~Ekedahl.
\newblock Subspace arrangements over finite fields: cohomological and
  enumerative aspects.
\newblock {\em Adv. Math.}, 129:159--187, 1997.

\bibitem[BSSW16]{BSSW}
L.~Bary-Soroker, Y.~Smilansky, and A.~Wolf.
\newblock On the function field analogue of {L}andau's theorem on sums of
  squares.
\newblock {\em Finite fields and their applications}, 39:195--215, 2016.

\bibitem[Cas]{Ca1}
K.~Casto.
\newblock {$\text{FI}_G$}-modules, orbit configuration spaces, and complex
  reflection groups.
\newblock arXiv:1608.06317, submitted.

\bibitem[CEF14]{CEF2}
T.~Church, J.~Ellenberg, and B.~Farb.
\newblock Representation stability in cohomology and asymptotics for families
  of varieties over finite fields.
\newblock {\em Contemp. Math.}, 620:1--54, 2014.

\bibitem[CEF15]{CEF}
T.~Church, J.~Ellenberg, and B.~Farb.
\newblock {FI}-modules: a new approach to stability for
  {$S_n$}-representations.
\newblock {\em Duke Math. J.}, 164(9):1833--1910, 2015.

\bibitem[CF13]{CF}
T.~Church and B.~Farb.
\newblock Representation theory and homological stability.
\newblock {\em Advances in Mathematics}, 245:250--314, 2013.

\bibitem[Che]{WCh2}
W.~Chen.
\newblock Twisted cohomology of configuration spaces and spaces of maximal tori
  via point-counting.
\newblock arXiv:1603.03931.

\bibitem[Chu12]{Chu}
T.~Church.
\newblock Homological stability for configuration spaces of manifolds.
\newblock {\em Invent. Math.}, 188(2):465--504, 2012.

\bibitem[Del80]{De}
P.~Deligne.
\newblock La conjecture de {W}eil: {II}.
\newblock {\em Inst. Hautes \'Etudes Sci. Publ. Math.}, 52:137--252, 1980.

\bibitem[EVW16]{EVW}
J.~Ellenberg, A.~Venkatesh, and C.~Westerland.
\newblock Homological stability for hurwitz spaces and the cohen-lenstra
  conjecture over function fields.
\newblock {\em Ann. of Math.}, 183:729--786, 2016.

\bibitem[FW]{FW2}
B.~Farb and J.~Wolfson.
\newblock {\'E}tale homological stability and arithmetic statistics.
\newblock arXiv:1512.00415.

\bibitem[Gad]{Ga2}
N.~Gadish.
\newblock Categories of {FI} type: a unified approach to generalizing
  representation stability and character polynomials.
\newblock arXiv:1608.02664.

\bibitem[GL15]{GL3}
W.L. Gan and L.~Li.
\newblock Coinduction functor in representation stability theory.
\newblock {\em J. London Math. Soc.}, 92(3):689--711, 2015.

\bibitem[Hu03]{Hu}
Y.~Hu.
\newblock A compactification of open varieties.
\newblock {\em Trans. Amer. Math. Soc.}, 355(12):4737--4753, 2003.

\bibitem[KL02]{KL}
M.~Kisin and G.~Lehrer.
\newblock Equivariant {P}oincar\'e polynomials and counting points over finite
  fields.
\newblock {\em J. Algebra}, 247(2):435–451, 2002.

\bibitem[Li09]{Li}
L.~Li.
\newblock Wonderful compactification of an arrangement of subvarieties.
\newblock {\em Michigan Math. J.}, 58(2):535--563, 2009.

\bibitem[Mat]{Mat}
V.~Matei.
\newblock A geometric perspective on {L}andau's problem over function fields.
\newblock In preparation.

\bibitem[MF94]{FM}
R.~Macpherson and W.~Fulton.
\newblock A compactification of configuration spaces.
\newblock {\em Ann. of Math.}, 139(1):183--225, 1994.

\bibitem[MP98]{MP}
R.~Macpherson and C.~Procesi.
\newblock Making conical compactifications wonderful.
\newblock {\em Sel. Math., New ser.}, 4:125--139, 1998.

\bibitem[OS80]{OS}
P.~Orlik and L.~Solomon.
\newblock Combinatorics and topology of complements of hyperplanes.
\newblock {\em Invent. Math.}, 56(2):167--189, 1980.

\bibitem[RW]{RW2}
R.J. Rolland and J.C.H. Wilson.
\newblock Stability for hyperplane complements of type {B}/{C} and statistics
  on squarefree polynomials over finite fields.
\newblock In preparation.

\bibitem[SS]{SS2}
S.~Sam and A.~Snowden.
\newblock Representations of categories of {$G$}-maps.
\newblock arXiv:1410.6054.

\bibitem[Wil14]{Wi2}
J.C.H. Wilson.
\newblock {$\text{FI}_\mathcal{W}$}-modules and stability criteria for
  representations of the classical {W}eyl groups.
\newblock {\em J. Algebra}, 420:269--332, 2014.

\end{thebibliography}

\end{document}